\documentclass[10pt]{manuscrit}
\def\currentvolume{1}
\def\currentissue{1}
\def\currentyear{2004}
\def\month{March}
\pagespan{1}{18}
\copyrightinfo{2004}{} 

\newtheorem{dummy}{dummy}[section]
\newtheorem{Theorem}[dummy]{Theorem}
\newtheorem{proposition}[dummy]{Proposition}
\newtheorem*{Proposition*}{Proposition}
\newtheorem{lemma}[dummy]{Lemma}

\newtheorem{defi}[dummy]{Definition}
\newtheorem{Remark}[dummy]{Remark}

\usepackage{graphicx}
\usepackage{hyperref}

\makeatletter
\AtBeginDocument{%
  \let\@serieslogo\@empty
}
\makeatother

\begin{document}

\title[Fourier/Hermite spectral methods for Vlasov-Poisson]
{On the convergence of Fourier/Hermite spectral methods with polynomial weights in velocity for the Vlasov-Poisson system\\}

\author[R. DAI]{Ruiyang DAI}
\address{
  Project-Team Makutu, Inria,
  University of Pau and Pays de l'Adour,
  Pau, France.
}
\email{ruiyang.dai@inria.fr}
\urladdr{https://makutu.fr/members/}






\abstract{
We analyze Fourier/Hermite spectral methods applied to
the Vlasov–Poisson (VP) system
based on time-independent asymmetrically-weighted (AW)
Hermite functions in velocity.
The VP system is written as an hyperbolic system
using AW Hermite functions in velocity.
We introduce a suitable weighted norm, with a time-independent polynomial weight, 
and obtain stability property for this norm.
This stability property provides a rigorous analysis for establishing error estimates of the methods.
Then we obtain error estimates between the exact solution and the numerical solution 
to the VP system.
}


\maketitle

\section{Introduction}

The VP system describes the microscopic dynamics of 
a collisionless plasma under the assumption that 
the electric field is electrostatic \cite{chen1984introduction,goldston2020introduction}. 
The system consists of the Vlasov equation, 
which governs the evolution of the particle distribution function in phase space, 
coupled with Poisson's equation, 
which determines the self--consistent electric field 
generated by the plasma charge density. 
This coupling is intrinsic: 
the particle distribution determines the charge density 
and hence the electric field, 
while the electric field influences the motion of the particles.

For ease of notation, we consider the normalized form of the VP system with boundary and initial conditions.
We consider a one-dimensional model of the plasma and only electrostatic forces, 
ignoring any electromagnetic effects.
One searches the electron distribution function $f(t,x,v)$ by solving 
the VP system of the plasma in dimensionless variables
\begin{equation}\label{eqn:vlasov_poisson_system}
\left\{
\begin{aligned}
& \frac{\partial f}{\partial t} 
 +v \frac{\partial f}{\partial x} 
 +E(t, x) \frac{\partial f}{\partial v} = 0, \\
& \frac{\partial E(t,x)}{\partial x} = \rho - \rho_0, \\
& f(t=0) = f_0, \\
\end{aligned}
\right.
\end{equation}
where $E(t,x)$ is the electrostatic field, 
and the density $\rho$ is given by
$$
  \rho(t, x) 
= \int_{\mathbb{R}} f(t, x, v) \,dv, 
  \quad t \geq 0, \,x \in\Omega_x = [0,x_{\text{max}}[. \\
$$
The constant \( \rho_0 \) ensures the quasi-neutrality condition of the plasma
\begin{equation}\label{eqn:quasi-neutrality-condition}
\int_{\Omega_x} (\rho - \rho_0) \, dx = 0.
\end{equation}

The VP system~\eqref{eqn:vlasov_poisson_system} has an intrinsic complexity 
due to the fact that the distribution function 
is defined in a six--dimensional phase space, 
comprising three spatial and three velocity dimensions. 
The system is also highly multiscale: 
plasma phenomena span a wide range of spatial and temporal scales, 
with several orders of magnitude of scale separation 
between microscopic and macroscopic scales. 
Consequently, developing numerical methods that can accurately 
capture the large--scale dynamics of collisionless plasmas 
while retaining the relevant microscopic physics 
remains a central challenge in computational plasma physics.

There are three major classes of numerical methods 
commonly used to solve the Vlasov equations, 
distinguished by how the plasma distribution function, 
or phase-space density, is represented. 
In Particle--In--Cell (PIC) methods, 
the plasma is represented by a finite number of computational macroparticles 
that move through a spatial mesh, 
while the self--consistent electric field 
is computed on the mesh \cite{birdsall2018plasma}. 
In Eulerian methods, 
the distribution function is directly discretized 
on a mesh covering both configuration and velocity space, 
resulting in a six-dimensional computational problem \cite{cheng1976integration,filbet2001conservative,filbet2003comparison}. 
Thirdly, spectral methods represent the velocity part of the distribution function using an appropriate set of basis functions. This leads to a system of evolution equations for the expansion coefficients, which depend only on the configuration-space variables \cite{engelmann1963nonlinear,armstrong1970solution}.


In this work, we are interested in spectral methods 
which provide an approach that can offer high-order accuracy in the numerical approximation of the VP system.
The expansion of the velocity variable in terms of Hermite functions 
reformulates the VP system 
\eqref{eqn:vlasov_poisson_system} as a hyperbolic system. 
The idea of representing the velocity dependence of the distribution function using polynomials multiplied by a Gaussian function, rather than directly discretizing the velocity variable, dates back to the 1960s \cite{armstrong1966numerical,joyce1971numerical}.
The use of Hermite functions within 
the framework of the Petrov--Galerkin method 
in the velocity variable offers several important advantages. 
Hermite functions are orthogonal with respect to a suitable weighted inner product 
and naturally account for the decay of the distribution function 
in the unbounded velocity domain, $-\infty < v < \infty$. 
Moreover, the recurrence relations satisfied by the Hermite functions 
result in a tridiagonal coupling structure for the resulting system, 
leading to a sparse and computationally efficient formulation.
For plasma physics simulations, authors in \cite{schumer1998vlasov} 
demonstrated that Hermite-based spectral methods 
can be efficient when an appropriate velocity scaling is employed.
Motivated by these results, we use rescaled, time-independent Hermite functions in the velocity variable. This choice provides an efficient representation of the velocity dependence while preserving the favorable structural properties of the Hermite basis.

In \cite{holloway1996spectral}, authors formalized two possible approaches using rescaled Hermite functions.
The first one is based on the symmetrically-weighted (SW) 
Hermite functions as the basis in velocity
and as test functions in the Galerkin method,
corresponding to the choice of weight function $\omega(v)=1$.
It shows numerically that this SW method cannot simultaneously conserve mass, momentum and total energy.
However, it conserves the $L^2$ norm of the distribution function, which ensures the stability of the method.
The second approach utilizes AW Hermite functions as trial functions
and a distinct set of test functions orthogonal to the AW Hermite basis
in the Petrov–Galerkin method.
In this case, the test functions are constructed
from AW Hermite functions together
with a weight function depending on $v$.
This approach yields the simultaneous conservation of
mass, momentum, and total energy.
However, it also shows numerically that the method based on AW Hermite functions
does not conserve the $L^2$ norm of the distribution function
and is then not numerically stable.
This raises the question of whether an error estimate can be established for the method based on AW Hermite functions.
For the spectral method based on time--independent SW Hermite functions, 
a convergence theory was proposed in \cite{manzini2017convergence}. 
For the spectral method based on time--dependent AW Hermite functions, 
recently, \cite{bessemoulin2023convergence} succeded in 
establishing a convergence theory 
with a time--dependent weight norm for the VP system. 
However, there is no error estimate of the classically 
time--independent weighted norm for the solution to the VP system.

To understand why there is no error estimate of the classically 
time--independent weighted norm for the solution to the VP system,
we consider the natural space
\[
  L^2(\omega dvdx) 
= \left\{ 
  g\in\mathbb{R}\times\mathbb{R}\,\Big|\, 
  \int_{\mathbb{R}^2} |g(x,v)|^2 \omega(v)\,dvdx < +\infty
  \right\}.
\]
The corresponding norm is
\[
  \|g\|_{L^2(\omega dvdx)}
= \left( \int_{\mathbb{R}^2} |g|^2\omega dvdx \right)^{\frac12}.
\]
Here, we define the weight as
\begin{equation}\label{eqn:weight_function_defi}
   \omega(v) 
 = \omega_\infty(v) 
:= \exp\left(\frac{v^2}{T}\right).
\end{equation}
Next, we compute the time derivative of $\|f(t)\|_{L^2(\omega dvdx)}$, 
$f$ being the solution of the VP system~\eqref{eqn:vlasov_poisson_system}. 
Using the Vlasov eqution in~\eqref{eqn:vlasov_poisson_system} 
and the definition~\eqref{eqn:weight_function_defi},
we have
\[
\begin{aligned}
   \dfrac12 \dfrac{\mathrm{d}}{\mathrm{d} t}
   \|f(t)\|_{L^2(\omega dvdx)}^2
&=-\int_{\mathbb{R}^2} f   \left( E \partial_v f \right)\,\omega\,dvdx \\
&= \dfrac1T \int_{\mathbb{R}^2} f^2 E\,v\,\omega\,dxdv. \\
&\leq
  \frac{\|E\|_{L^\infty}}{T} \int_{\mathbb{R}^2} f^2 |v|\,\omega\,dvdx.
\end{aligned}
\]
Obviously, there is no constant $C$ such that  
\[
  \frac{\|E\|_{L^\infty}}{T} \int_{\mathbb{R}^2} f^2 |v|\,\omega\,dvdx
\leq C
  \|f(t)\|_{L^2(\omega dvdx)}^2.
\]
for all $f$.
Consequently, we can not apply Gr\"onwall's Lemma.
Indeed, the energy estimate involves the additional factor $|v|$, 
which cannot be controlled by $\|f(t)\|_{L^2(\omega dvdx)}^2$ alone. 

To overcome this difficulty, 
we consider another weighted norm
\[
  \|g\|_{L^2(\sigma dvdx)}
= \left( \int_{\mathbb{R}^2} |g|^2\sigma dvdx \right)^{\frac12}.
\]
Let $M\geq 0$ be an integer and define
\begin{equation}\label{eqn:another_weight_function_defi}
  \sigma(v) 
 =\omega_M(v) 
:=\sum_{m=0}^M \dfrac{v^{2m}}{m!}.
\end{equation}
We notice that the limiting weight of $\omega_M(v)$ is $\omega(v)$.
It is easy to prove that, 
if $M\ge 1$, then
for every $v\in\mathbb R$,
$ |v|\,\omega_{M-1}(v) \leq \sqrt{M}\,\omega_M(v) $.
The stability result follows directly
\[
\begin{aligned}
   \dfrac12 \dfrac{\mathrm{d}}{\mathrm{d} t}
   \|f(t)\|_{L^2(\omega_M dvdx)}^2
&\leq
  \frac{\|E\|_{L^\infty}}{T} \int_{\mathbb{R}^2} f^2 |v|\,\omega_{M-1}\,dvdx \\
&\leq C
  \sqrt{M} \|f(t)\|_{L^2(\omega_M dvdx)}^2. \\
\end{aligned}
\]
Finally,
applying Gr\"onwall's lemma, we obtain
\[
  \|f(t)\|_{L^2(\omega_M\,dvdx)}
\leq
  e^{C\sqrt{M}\,t} \|f(0)\|_{L^2(\omega_M\,dvdx)},
  \quad t\geq 0.
\]

In \cite{dai2026galerkin}, 
authors provide a stable spectral method based on AW Hermite functions
in the standard $L^2$ space for the VP system. 
The main idea is to use the Galerkin method instead of 
the Petrov-Galerkin method which is quadratically stable with
respect to the standard $L^2$ norm.
The convergence of the stable method can be analyzed within the framework 
in this work,
since the standard $L^2$ space for the VP system is equivalent to 
the polynomially-weighted $L^2$ space when $M=0$.
In this work, we don't discuss the case $M=0$, 
since additional control or a separate arguement is necessary (see Remark~\eqref{remark:important_remark}).

The contents of this paper are organized as follows. 
In Section \ref{sec:preliminaries}, we introduce the necessary preliminaries. 
Section \ref{sec:hermite-approximations} is devoted to Hermite approximations, 
where we study both the orthogonal and non--orthogonal projections 
and establish the approximation results needed for the subsequent analysis. 
In Section \ref{sec:fourier-approximations}, 
we consider Fourier approximations and 
present the corresponding approximation properties. 
Section \ref{sec:hermite-fourier-spectral-method} presents 
the Hermite/Fourier spectral method, 
including the Hermite spectral discretization in the velocity variable and 
the Fourier spectral discretization in the spatial variable.
The convergence is carried out in Section \ref{sec:convergence}. 
We first derive the global projection error, 
then analyze the consistency error, 
and finally combine these estimates to establish the main convergence result. 
Finally, Section \ref{sec:conclusion} concludes the paper. 

\section{Preliminaries}\label{sec:preliminaries}

We introduce a Gaussian function $y(v)$, satisfying
\[
0 < y(v) < \infty, \qquad
0 \le \int_{\mathbb{R}} v^n y(v)\,dv < \infty,
\quad \forall n \in \mathbb{N}.
\]
We define the weighted polynomial space 
in terms of a weight function $\omega$
\[
V^{\omega} = \operatorname{span}\{v^n y(v)\}_{n \in \mathbb{N}},
\]
which constitutes an infinite-dimensional linear space, equipped with the weighted inner product
\[
(g,h)_{\omega dv} := \int_{\mathbb{R}} g(v) h(v) \omega(v)\,dv < \infty,
\qquad g,h \in V^\omega.
\]
We introduce the second inner product in terms of another weight function $\sigma(v)$ 
\begin{equation}\label{eqn:inner_product_L2}
(g,h)_{\sigma dv} := \int_{\mathbb{R}} g(v) h(v) \sigma(v)\,dv,
\qquad g,h \in V^\omega.
\end{equation}
We assume that 
$(g,h)_{\omega dv} < \infty$ implies $(g,h)_{\sigma dv} < \infty$.

For a non-negative integer $N \in \mathbb{N}$, let $V_N^{\omega}$ 
be a closed subspace of $V^\omega$ with $\dim(V_N^{\omega}) = N+1$.
Let $\{\phi_0,\phi_1,\ldots,\phi_k,\ldots\}$ be a basis of $V^{\omega}$ and
$\{\varphi_0,\varphi_1,\ldots,\varphi_N\}$ be a basis of $V_N^{\omega}$, respectively.
Since $V_N^{\omega}$ is a subspace of $V^{\omega}$, there exists a matrix
$P_b \in \mathbb{R}^{(N+1)\times\infty}$ with full row rank such that
\[
(\varphi_0,\varphi_1,\ldots,\varphi_N)^\top = P_b (\phi_0,\phi_1,\ldots,\phi_k,\ldots)^\top.
\]

A linear bounded operator $\mathcal{P}:V^{\omega} \to V_N^{\omega}$ is called a projection
operator on $V_N^{\omega}$ if
$\mathcal{P}g \in V_N^{\omega}$ for all $g \in V^{\omega}$, and
$\mathcal{P}g = g$ for all $g \in V_N^{\omega}$.

For any $g \in V^{\omega}$, there exist coefficients $\widehat{g}_i$, $i=0,1,\ldots$, such that
\[
g = \sum_{i=0}^\infty \widehat{g}_i \phi_i.
\]
Since $\mathcal{P}g \in V_N^{\omega}$, there exist $\bar g_i$, $i=0,\ldots,N$, such that
\[
\mathcal{P}g = \sum_{i=0}^N \bar g_i \varphi_i.
\]
Since $\mathcal{P}$ is linear, there exists a unique matrix
$P_p \in \mathbb{R}^{(N+1)\times\infty}$ satisfying
\[
  (\bar{g}_0,\bar{g}_1,\ldots,\bar{g}_N)^\top 
= P_p \, (\widehat{g}_0,\widehat{g}_1,\ldots,\widehat{g}_k\ldots)^\top 
\]
such that
\begin{equation}
  \mathcal{P}g 
= \left\langle 
  P_b \,(\phi_0,\phi_1,\ldots,\phi_k,\ldots)^\top, 
\,P_p \,(\widehat{g}_0,\widehat{g}_1,\ldots,\widehat{g}_k\ldots)^\top 
  \right\rangle,
\end{equation}
where $\langle\cdot,\cdot\rangle$ denotes the inner product of vectors.

Since $\mathcal{P}$ is a linear bounded projection operator, we have
\[
\|P_p\| < \infty,
\qquad
P_b \,P_p^\top = I_{N+1},
\]
where $\|\cdot\|$ is a matrix norm and $I_{N+1}$ is the $(N+1)\times (N+1)$ identity matrix.
Clearly, the projection operator $\mathcal{P}$ is uniquely determined by $P_p$, 
and thus we may directly use $P_p$ to denote the projection on the weighted polynomial space.

In particular, for the classical orthogonal projection, i.e.,
\[
\|\mathcal{P}g - g\|_{L^2(\omega dv)} \le \|f - g\|_{L^2(\omega dv)}, 
\qquad \forall f \in V_N^{\omega},
\]
we have
\begin{equation}
P_p
=
\left(\left( \varphi_i,\varphi_j\right)_{\omega dv}\right)^{-1}_{(N+1)\times(N+1)}
\, P_b \,
\left(\left( \phi_i,\phi_j \right)_{\omega dv}\right)_{\infty\times\infty}.
\end{equation}
Furthermore, if $\varphi_i = \phi_i$ for $i=0,\ldots,N$ and
\[
(\varphi_i,\phi_j)_{\omega dv} = 0,
\quad \forall i=0,\ldots,N,\; j=N+1,N+2,\ldots,
\]
then the orthogonal projection reduces to a cut-off operator and
\[
P_b = P_p = \left(\, I_{N+1} \;\; 0 \,\right).
\]

It is also possible to consider the $L^2(\sigma dv)$ inner product of 
the infinite-dimensional linear space $V^\omega$ when $f \in V^\omega$.
Specifically, we have
\begin{equation}
P_p
=
\left(\left( \varphi_i,\varphi_j\right)_{\sigma dv}\right)^{-1}_{(N+1)\times(N+1)}
\, P_b \,
\left(\left( \phi_i,   \phi_j   \right)_{\sigma dv}\right)_{\infty\times\infty}.
\end{equation}

Since the approximation method considered below 
is formulated in the finite-dimensional space obtained by truncating at $N$, 
we partition any infinite matrix $Q = (q_{nm})_{n,m \geq 0}\in\mathbb{R}^{\mathbb{N}\times\mathbb{N}}$ 
and vector $Y = (y_n)_{n\geq 0}\in\mathbb{R}^{\mathbb{N}}$ according to
\[
\renewcommand{\arraystretch}{1.5}
Q =
\begin{bmatrix}
Q_{11}^N & Q_{12}^N \\
Q_{21}^N & Q_{22}^N \\
\end{bmatrix},
\qquad
Y =
\begin{bmatrix}
Y_{1}^N \\
Y_{2}^N \\
\end{bmatrix},
\]
where
\[
\begin{aligned}
Q_{11}^N &= (q_{nm})_{0\leq n,m\leq N}, &&
Q_{12}^N  = (q_{nm})_{0\leq n\leq N, \, N+1\leq m}, \\
Q_{21}^N &= (q_{nm})_{N+1\leq n, \, 0\leq m\leq N}, &&
Q_{22}^N  = (q_{nm})_{n,m\geq N+1},
\end{aligned}
\]
and
\[
Y_1^N = (y_n)_{0\leq n\leq N},
\qquad
Y_2^N = (y_n)_{n\geq N+1}.
\]
Unless otherwise specified,
all blocks $Q_{11}, Q_{12}, Q_{21}, Q_{22}$ and $Y_{1}, Y_{2}$
are understood to be associated with the truncation $N$,
and the superscript $\cdot^N$ is omitted throughout to simplify the notation.

\section{Hermite approximations}\label{sec:hermite-approximations}

Several types of Hermite approximations can be considered, 
depending on the choice of the basis functions 
and the weight function \( \omega(v) \).
A classical approach is to use 
the standard Hermite polynomials as basis functions; 
see, for example,
\cite{szeg1939orthogonal,gottlieb1977numerical,hussaini1986spectral,bernardi1997spectral,guo1999error}. 
In this case, \( y(v) = 1 \), 
and to obtain the orthogonal system \( \omega(v) = \exp(-v^2) \). 
The second approach is to use the Hermite functions, 
obtained by combining Hermite polynomials with a Gaussian function, 
as base functions. 
In \cite{funaro1991approximation}, 
Hermite polynomials with \( y(v)=\exp(-v^2/(4a^2)) \) 
are employed as basis functions, 
with the corresponding weight function given by \( \omega(v)=\exp(v^2/(4a^2)) \).
The authors established approximation results 
for a class of diffusion evolution equations.
Similarly, in \cite{fok2002combined}, 
Hermite polynomials with \( y(v) = \exp(-\alpha^2 v^2) \)
are used as basis functions, 
together with the weight function \( \omega(v) = \exp(\alpha^2 v^2) \). 
Approximation results were obtained in the analysis of 
the spectral method based on AW Hermite functions for the Fokker--Planck equation.

In general, the choice of \( y(v) \) depends on 
the asymptotic behaviour of the solutions of the considered problems. 
In many problems arising in plasmas physics, 
the solutions decay exponentially as \( |v| \to \infty \). 
In this case, it is reasonable to take the basis functions 
as those used in \cite{funaro1991approximation,fok2002combined}.

Following the second approach,
for the VP system, 
\cite{manzini2017convergence} employed 
the SW Hermite functions 
(the Gaussian function is \( y(v) = \exp(-v^2/(4a^2) \)) 
as basis functions with the weight function  
\( \omega(v) = \exp(v^2/(4a^2)) \). 
The authors derived approximation results 
for the spectral method with SW Hermite functions applied to the VP system.
More recently,
\cite{bessemoulin2023convergence} considered 
AW Hermite functions (\( y(v) = \exp( -(\alpha(t) v)^2/2) \))
where the scaling parameter \( \alpha(t) \) is time-dependent.
The corresponding time-dependent 
weight function is \( \omega(v) = \exp( (\alpha(t) v)^2/2 ) \). 
The~\cite{bessemoulin2023convergence} work established the first approximation result 
for the spectral method with time--dependent AW Hermite functions 
applied to the VP system.

The stability properties of Hermite-based discretizations for the VP system 
have also been investigated. 
In \cite{dai2024quadratic}, 
the authors analyzed the instability that 
may arise in the discretization of linear transport equations 
based on AW Hermite functions, 
with particular relevance to numerical methods 
for collisionless plasma physics. 
In \cite{dai2026galerkin}, the same issue was investigated for the VP system. 
However,
to the best of our knowledge, 
there is no estimate currently available for 
the associated time-independent weighted norm of solutions 
to the VP system.
Motivated by these works, we study the error estimate of the Galerkin method
based on AW Hermite functions for the VP system. 
In contrast to the aforementioned weighted formulations, 
we consider the polynomially-weighted inner product.

Let ${H}_n(v)$ be the Hermite polynomial \cite{nist} of degree $n$,
\begin{equation}
{H}_n(v) = (-1)^n e^{v^2} \dfrac{d^n}{dv^n}(e^{-v^2}),
\end{equation}
which is orthogonal on $(-\infty, \infty)$ with respect to the Gaussian function $\exp(-v^2)$.
We choose 
$$
y(v) = \exp(-\frac{v^2}{T})
$$
as the Gaussian function, 
where $T > 0$ is a constant related to the thermal velocity.
Accordingly,
the AW Hermite functions $\psi_n(v)$ are given by
\begin{equation}
  \psi_n(v) 
= \left( {2^n n! \sqrt{\pi T}} \right)^{-\frac12}
  {H}_n(\frac{v}{\sqrt{T}}) \exp(-\frac{v^2}{T}),
\quad n \geq 0.
\end{equation}
Next, we choose 
\[
  \omega(v) 
= \omega_{\infty}(v),
\]
and the set of functions $(\psi_n(v))_n$ is an orthogonal system satisfying
$$
  \left( \psi_n, \psi_m \right)_{\omega dv}
= \int_{\mathbb{R}} \psi_n(v) \psi_m(v) \omega(v)\,dv 
= \delta_{nm},
$$
where $\delta_{nm}$ is the Kronecker delta function.
The function $\psi_n(v)$ is the $n$-th eigenfunction of the following singular 
Liouville problem:
\begin{equation}
\dfrac{d}{dv} 
\left( 
  e^{\frac{-v^2}{T}} \dfrac{d}{dv} 
  \left( 
  e^{\frac{ v^2}{T}} f(v)
  \right)
\right)
+ \lambda f(v)
= 0, \quad v \in \mathbb{R}.
\end{equation}
The corresponding eigenvalues are $\lambda_n = \frac{2n}{T}$.
It can be shown that, $\forall n \geq 0$,
\begin{equation}
\begin{aligned}
  v \psi_n(v) 
&=\sqrt{T}
  \left(
    \sqrt{\dfrac{n+1}{2}} \psi_{n+1}(v) 
  + \sqrt{\dfrac{n}{2}}   \psi_{n-1}(v) 
  \right), \\
  \dfrac{d\psi_n(v)}{dv}
&=\dfrac{-1}{\sqrt{T}} \sqrt{2(n+1)} \psi_{n+1},
\end{aligned}
\end{equation}
where $\psi_n(v) := 0$ for $n < 0$.

Let the weighted polynomial space 
$V^\omega = \operatorname{span}\{\psi_n \}_{n \in \mathbb{N}}$,
equipped with the norm
\[
  \left|\left| g \right|\right|_{L^2(\omega dv)}
= \left( \int_{\mathbb{R}} |g(v)|^2 \omega(v)\,dv \right)^{\frac12}. 
\]
The second weighted $L^2(\sigma dv)$ norm is
\[
  \left|\left| g \right|\right|_{L^2(\sigma dv)}
= \left( \int_{\mathbb{R}} |g(v)|^2 \sigma(v)\,dv \right)^{\frac12}. 
\]
For any non-negative integer $r$,
\[
H^r(\omega dv) = \left\{ g \big| \dfrac{d^k g}{d v^k} \in L^2(\omega dv), 0 \leq k \leq r \right\}
\]
with the following semi-norm
\[
\left| g \right|_{H^r(\omega dv)} = \left|\left| \dfrac{\partial^r g}{\partial v^r} \right|\right|_{L^2(\omega dv)}.
\]
and norm
\[
  \left|\left| g \right|\right|_{H^r(\omega dv)}
= \left( \sum_{k=0}^r |g|^2_{H^k(\omega dv)} \right)^{\frac12}.
\]
For any $g \in V^\omega$, we can expand $g$ in the following form:
\begin{equation}
g(v) = \sum_{n=0}^{\infty} \widehat{g}_n \psi_n(v)
\end{equation}
with the Hermite coefficients
$$
\widehat{g}_n = \int_{\mathbb{R}} g(v) \psi_n(v) \omega(v) dv, \quad n \geq 0.
$$
Next, we choose a non negative integer $N$ and 
let $V_N^\omega = \operatorname{span}\{\psi_n \}_{n \leq N}$ 
be the finite-dimensional weighted polynomial space.

\subsection{Orthogonal projection}
We consider the orthogonal projection 
$\mathcal{P}: V^\omega \to V_N^\omega$ that is a mapping such that,
$\forall g \in V^\omega$,
\[
(g - \mathcal{P} g, \phi)_{\omega dv} = 0, \quad \phi \in V_N^\omega.
\]
We have $P_p = P_b = \left( I_{N+1} \ 0 \right)$.
Equivalently,
\[
  \mathcal{P} g(v) 
= \sum_{n=0}^N \sum_{m=0}^\infty \left( P_p\right)_{nm} \widehat{g}_m \psi_n(v)
= \sum_{n=0}^N \widehat{g}_n \psi_n(v).
\]
We also introduce the operator $\mathcal{A}$ as
\begin{equation}
  \mathcal{A}g(v)
=-\dfrac{d}{dv} 
  \left( 
    e^{-\frac{v^2}{T}} 
    \dfrac{d}{dv}
    \left( g(v)e^{\frac{v^2}{T}} \right)
  \right).
\end{equation}
\begin{Theorem}
$\forall g \in H^r(\omega dv)$ and $r \geq 0$,
\begin{equation}
|| g - \mathcal{P} g ||_{L^2(\omega dv)}
\leq C \left( \dfrac{N}{T} \right)^{-\frac{r}{2}} 
|| g ||_{H^r(\omega dv)}. 
\end{equation}
\end{Theorem}
\begin{proof}
The result follows directly from the approximation estimate 
established in \cite[Theorem 2.1]{fok2002combined}. 
\end{proof}
\begin{Theorem}
$\forall g \in H^r(\omega dv)$ and $0 \leq s \leq r$,
\begin{equation} 
|| g - \mathcal{P} g ||_{H^{s}(\omega dv)} 
\leq C \left( \dfrac{N}{T} \right)^{\frac{s-r}{2}} 
|| g ||_{H^r(\omega dv)}. 
\end{equation}
\end{Theorem}
\begin{proof}
The result follows directly from the approximation estimate 
established in \cite[Theorem 2.2]{fok2002combined}. 
\end{proof}

\subsection{Non-orthogonal projection}
We choose
\[
  \sigma(v)
= \omega_M(v).
\]
For $M\to\infty$, 
\[
\omega_M(v)\to\sum_{m=0}^\infty \dfrac{v^{2m}}{m!}=\omega(v).
\]
$\forall g, h \in V^\omega$, since $(g,h)_{\omega dv} < \infty$ implies $(g,h)_{\omega_M dv} < \infty$, 
it is possible to consider the projection $\mathcal{P}$ 
with the $L^2(\omega_M dv)$ inner product
\begin{equation}\label{eqn:hermite_non_orthogonal_projection}
(g-\mathcal{P}g, \phi)_{\omega_M dv} = 0, \quad \phi \in V_N^\omega.
\end{equation}
Since the projection operator $\mathcal{P}$ is not orthogonal 
in terms of the inner product \eqref{eqn:inner_product_L2}, we have 
\begin{equation}
\begin{aligned}
   P_b 
&= \left( I_{N+1} \ 0\right), \\
   P_p 
&= 
   \left(\left( \psi_i,\psi_j \right)_{\omega_M dv}\right)^{-1}_{(N+1)\times(N+1)}
   \, P_b \,
   \left(\left( \psi_i,\psi_j \right)_{\omega_M dv}\right)_{\infty\times\infty} \\
&= 
   \left(\left( \psi_i,\psi_j \right)_{\omega_M dv}\right)^{-1}_{(N+1)\times(N+1)}
   \left(\left( \psi_i,\psi_j \right)_{\omega_M dv}\right)_{(N+1)\times\infty}.
\end{aligned}
\end{equation}
Equivalently,
\begin{equation}\label{eqn:non_orthogonal_projection_expansion}
  \mathcal{P} g(v) 
= \sum_{n=0}^N \sum_{m=0}^\infty \left( P_p\right)_{nm} \widehat{g}_m \psi_n(v).
\end{equation}

We introduce the following definition which is yield by the non-orthogonal projection.
\begin{defi}[Gram matrix]
The infinite Gram matrix
$A^M = (a_{mn}^M)_{m,n \geq 0} \in \mathbb{R}^{\mathbb{N} \times \mathbb{N}}$ of
the problem is the collection of all $L^2(\omega_M dv)$ scalar products of the AW functions. 
The coefficient is
$$
a_{mn}^M = \int_{\mathbb{R}} \psi_m(v) \psi_n(v) \omega_M(v) dv.
$$
Unless otherwise specified, 
all entries $(a_{mn})_{m,n \geq 0}$ and 
the Gram matrix $A$ are understood to be associated with 
the finite number $M$, 
and the superscript $\cdot^M$ is omitted throughout to simplify the notation.
\end{defi}

\begin{Theorem}\label{theo:hermite_non_ortho_projection_error}
$\forall g \in H^r(\omega dv)$ and $r \geq 0$, 
\begin{equation}
|| g - \mathcal{P} g ||_{L^2(\omega_M dv)}
\leq C \left( \dfrac{N}{T} \right)^{-\frac{r}{2}} 
|| g ||_{H^r(\omega dv)}. 
\end{equation}
\end{Theorem}
\begin{proof}
It follows from the Hermite expansion \eqref{eqn:non_orthogonal_projection_expansion}, 
for any $g \in H_\omega^r(\mathbb{R})$, 
that
\[
\begin{aligned}
  g - \mathcal{P} g 
&=\sum_{n=0}^{\infty} \widehat{g}_n \psi_n(v) 
 -\sum_{n=0}^{N} 
  \sum_{m=0}^{\infty} (P_p)_{nm} \widehat{g}_m 
  \psi_n(v) \\ 
&=\sum_{n=N+1}^{\infty} \widehat{g}_n \psi_n(v) 
 +\sum_{n=0}^{N} \left( 
  \widehat{g}_n
 -\sum_{m=0}^{\infty} (P_p)_{nm} \widehat{g}_m 
  \right) 
  \psi_n(v) \\ 
\end{aligned}
\]
The infinite vector of $\widehat{g}_n$ is denoted as
$\widehat{G} = (\widehat{g}_n)_{n \geq 0} \in \mathbb{R}^{\mathbb{N}}$. 
Then we take the block vector by truncation $N$ of $\widehat{G}$, 
and the block matrix by truncation $N$ of $A$.
Considering its $L^2$ norm, we have
\begin{equation}\label{eq:hermite-approximation-equation}
  || g - \mathcal{P} g ||_{L^2(\omega_M dv)}^2
= \widehat{G}_2^{\top} 
  \left( A_{22} - A_{21} A_{11}^{-1} A_{12} \right) 
  \widehat{G}_2. 
\end{equation}
Since $A$ and $A_{11}$ are symmetric positive definite, 
the Schur complement satisfies
\[
A_{22} - A_{21} A_{11}^{-1} A_{12} \ge 0. 
\]
We also have 
\[
0 \le A_{22} - A_{21} A_{11}^{-1} A_{12} \le A_{22}. 
\]
Indeed, $A_{21} A_{11}^{-1} A_{12}$ is positive semidefinite.
Let $X$ be a finitely supported sequence, and define
\[
h(v) = \sum_{n=0}^\infty x_n \psi_n(v).
\]
Take its tail sequence $X_2=(x_{N+1},x_{N+2},\dots)$.
Then
\[
  X_2^\top A_{22} X_2 
\le 
  \left\| \sum_{n=0}^\infty x_n \psi_n(v) \right\|_{L^2(\omega_M dv)}
\le 
  \left\| \sum_{n=0}^\infty x_n \psi_n(v) \right\|_{L^2(\omega dv)}
= X^\top X. 
\]
Thus
\[
  X_2^\top 
  \left( A_{22} - A_{21} A_{11}^{-1} A_{12} \right) 
  X_2 
\le X^\top X. 
\]
Using \eqref{eq:hermite-approximation-equation}, we have
\[
\begin{aligned}
  || g - \mathcal{P} g ||_{L^2(\omega_M dv)}^2
&\le \sum_{n=N+1}^\infty |\widehat{g}_n|^2 \\
&\le \, C \left(\dfrac{N+1}{T}\right)^{-r}
     \sum_{n=N+1}^{\infty} 
     \left| \int_{\mathbb{R}} 
     \mathcal{A}^{\frac{r}{2}} g(v) \psi_n(v) \omega(v) dv \right|^2, \\
\end{aligned}
\]
where
\[
  |\widehat{g}_n|
= \left(\dfrac{2n}{T}\right)^{-\frac{r}{2}}
  \left| \int_{\mathbb{R}} \mathcal{A}^{\frac{r}{2}} g(v) \psi_n(v) \omega(v) dv \right|,
  \quad n > N
\]
as given in \cite[Eq.~(2.21)]{fok2002combined}.
Finally,
\[
|| g - \mathcal{P} g ||_{L^2(\omega_M dv)} 
\leq C \left( \dfrac{N}{T} \right)^{-\frac{r}{2}} 
|| g ||_{H^r(\omega dv)}. 
\]
\end{proof}
\begin{lemma}\label{lem:lemma-non-orthogonal-projection}
Let $M \geq 0$ be an integer.
If $g\in H^1(\omega dv)$, then, there exists a constant $C$ such that
\[
\left\{
\begin{aligned}
& \|g\|_{L^2(\omega_{M+1} dv)} \leq \sqrt{2C} |g|_{H^1(\omega dv)}, \\
& \|v\,g\|_{L^2(\omega_M dv)}  \leq C |g|_{H^1(\omega dv)}. \\
\end{aligned}
\right.
\]
\end{lemma}
\begin{proof}
Using the fact that $\partial_v \omega_{M+1}(v) = 2v\omega_M(v)$, we obtain
\[
  \int_{\mathbb{R}} (v\,g(v))^2 \omega_M(v)\,dv
= \dfrac12 
  \int_{\mathbb{R}}  v\,g(v)^2 \partial_v\omega_{M+1}(v)\,dv.
\]
By an integration by parts, we obtain
\[
  \int_{\mathbb{R}} (v\,g(v))^2 \omega_M(v)\,dv
=-\dfrac12 
  \int_{\mathbb{R}}  g(v)^2 \omega_{M+1}(v)\,dv 
 -\int_{\mathbb{R}}  v\,g(v) \partial_vg(v) \omega_{M+1}(v)\,dv.
\]
Then, applying Cauchy-Schwarz inequality to the second term of 
the right-hand side of the equation above, it yields
\begin{equation}\label{eqn:equation-in-lemma-non-orthogonal-projection}
  \|v\,g\|^2_{L^2(\omega_M dv)}
+ \dfrac12
  \|g\|^2_{L^2(\omega_{M+1} dv)}
\leq
  \|v\,g\|_{L^2(\omega_{M+1} dv)}
  |g|_{H^1(\omega_{M+1} dv)}.
\end{equation}
We can discard the second therm of the left-hand side of~\eqref{eqn:equation-in-lemma-non-orthogonal-projection},
then divide both sides by $\|v\,g\|_{L^2(\omega_{M} dv)}$. We obtain
\[
  \|v\,g\|_{L^2(\omega_M dv)}
\leq C
  |g|_{H^1(\omega_{M+1} dv)}
\leq C
  |g|_{H^1(\omega dv)}.
\]
Next, we discard the first term of the left-hand side of~\eqref{eqn:equation-in-lemma-non-orthogonal-projection},
using the last estimate, we obtain 
\[
  \dfrac12 \|g\|^2_{L^2(\omega_{M+1} dv)}
\leq C
  |g|_{H^1(\omega_{M+2} dv)}
  |g|_{H^1(\omega_{M+1} dv)}
\leq C
  |g|_{H^1(\omega dv)}^2,
\]
which completes the proof.
\end{proof}

\section{Fourier approximations}\label{sec:fourier-approximations}

We consider the Fourier space
\[
F = \operatorname{span}\{ e_n \}_{n \in \mathbb{Z}}.
\]
For periodic functions, we can equivalently express them in their complete form
\[
  g(x) 
= \sum_{n \in \mathbb{Z}} 
  c_n \left( \frac{1}{\sqrt{|\Omega_x|}} 
  \exp\left(\imath \frac{2\pi n x}{|\Omega_x|}\right) \right), 
  \quad x \in \Omega_x.
\]
We define periodic Sobolev spaces $H^r_{\text{per}}(dx)$, norms, and seminomes in the same way as $H^r(\omega dv)$ spaces.
In addition, 
the $L^2$ weighted space given by
\[
W := F \otimes V^{\omega},
\]
with the associate inner product
\[
  \left( g, h \right)_{\omega dvdx} 
= \int_{\Omega_x} \int_{\mathbb{R}} g(x,v) h(x,v) \omega dvdx,
\]
and $||\cdot||_{L^2(\omega dvdx)}$ the corresponding norm.

Next, we define the orthogonal projection $\mathcal{Q} : F \to F_J$ 
that is a mapping such that \( \forall g \in F \), 
\[
\left(g - \mathcal{Q} g, \phi \right)_{dx} = 0, \quad \forall \phi \in F_J,  
\]
where the inner product on $F$ is given by
\[
  \left( g, h \right)_{dx} 
= \int_{\Omega_x} g(x) \,\overline{h(x)} \,dx, \quad g, h \in F.
\]
It follows that $\mathcal{Q}$ is the projection on the finite-dimensional space
\[
F_J = \operatorname{span}\{ e_n \}_{n = -J,\dots,J}
\]
which consists of trigonometric polynomials of degree less than 
or equal to $J$ and of period $|\Omega_x|$.
More over, we have 
\[
\dim \left( F_J \right) = 2J+1.
\]

We now recall a classical result concerning the Fourier approximation in the context of spatial discretization.
\begin{proposition}\label{prop:fourier_projection_error}
Let $s>0$ be a integer. For all $g \in H_{\text{per}}^s(dx)$, we have
\begin{equation}
\left|\left| g - \mathcal{Q} g \right|\right|_{L^2(dx)} \leq C J^{-s} \left| g \right|_{H_{\text{per}}^s(dx)}, 
\quad J > 0, 
\quad C = \left( \frac{|\Omega_x|}{2\pi} \right)^{s}.
\end{equation}
\end{proposition}

\section{Hermite/Fourier spectral method}\label{sec:hermite-fourier-spectral-method}

In this section, we introduce the Hermite/Fourier spectral method. 
First, we address the velocity discretization by representing 
the approximation of the distribution function $f$ with AW Hermite functions. 
Next, we handle the spatial discretization using the Fourier functions.

For convenience, let 
$ B = (b_{mn})_{m,n\geq 0} \in\mathbb{R}^{\mathbb{N}\times\mathbb{N}}$ 
denote the infinite transport matrix, defined by
\[
b_{mn} =
\left\{
\begin{aligned}
&  \sqrt{m}, && m - n = 1,\\
&  \sqrt{n}, && m - n =-1,\\
&  0, && \text{otherwise}.\\
\end{aligned}
\right.
\]
Thus, $B$ is symmetric and tridiagonal. 
Similarly, let 
$D=(d_{mn})_{m,n\geq 0}\in\mathbb{R}^{\mathbb{N}\times\mathbb{N}}$ 
denote the infinite acceleration matrix, with entries
\[
d_{mn} =
\left\{
\begin{aligned}
&  \sqrt{m}, && m - n = 1,\\
&  0, && \text{otherwise}.\\
\end{aligned}
\right.
\]
Hence, $D$ is strictly lower triangular and 
has nonzero entries only on its first subdiagonal.

\subsection{Hermite spectral discretization}

We look for an approximation $f_N$ of the VP system \eqref{eqn:vlasov_poisson_system} 
as a finite sum which corresponds to a truncation of a series
\[
f_N(t,x,v) = \sum_{n=0}^{N} u_n(t,x)\psi_n(v).
\]
The unknowns $(u_n)_{0\leq n \leq N}$ are computed 
using the $L^2(\omega_M dv)$ inner product 
and taking $(\psi_n(v))_{0\leq n \leq N}$ as test functions in \eqref{eqn:vlasov_poisson_system}.
Therefore, we obtain a system of evolution equations for 
the unknowns $(u_n)_{0\leq n \leq N}$ :
\begin{equation}\label{eqn:approximation_fN}
\begin{aligned}
   \sum_{m=0}^{N} a_{nm} \partial_t u_m
+& \sqrt{\frac{T}2}
   \sum_{m=0}^{N}
   \left( \sqrt{m}  a_{n,m-1}
	+ \sqrt{m+1}a_{n,m+1}
   \right) \partial_x u_m \\
-& \sqrt{ \frac{2}T} E_N
   \sum_{m=0}^{N} \sqrt{m+1} a_{n,m+1} u_m
=0.
\end{aligned}
\end{equation}
The system above can also be written as
\[
  \left\langle A \partial_t U, V \right\rangle
+ \sqrt{\dfrac{T}{2}} \left\langle A B \partial_x U, V \right\rangle
- \sqrt{\dfrac{2}{T}} E_N \left\langle A D U, V \right\rangle
= 0,
\]
where 
\[
U,V \in \,\operatorname{supp}(V_N^{\omega})
:= \left\{ (u_m) \in \mathbb{R}^{\mathbb{N}} \,\big|\, 
u_m = 0 \text{ for } m > N \right\}.
\]
The initial data $u_n(t=0)$ is given by
\[
  \left(f_N(t=0), \psi_n\right)_{L^2(\omega_M dv)}
= \left(f_0, \psi_n\right)_{L^2(\omega_M dv)}.
\]
Meanwhile, we observe that the density $\rho_N$ satisfies 
\[
\rho_N = \int_{\mathbb{R}} f_N dv = (\pi T)^{\frac14} u_0
\]
and then the Poisson equation becomes
\[
\dfrac{\partial E_N}{\partial x}  = (\pi T)^{\frac14} u_0 - \rho_{0,N},
\]
with $\rho_{0,N}$ such that 
\[
\int_{\Omega_x} \left( \rho_N - \rho_{0,N} \right) dx = 0.
\]

\subsection{Fourier spectral discretization}

We consider the Fourier spectral discretization for the Vlasov equation 
with AW Hermite functions in velocity.
We consider the subspaces $F_J$ and $V_N^{\omega}$, with the parameter 
$\tau = (1/J, 1/N)$ related to the numerical discretization in space and velocity.
We look for an approximation $u_{\tau,n}(t,\cdot) \in F_J$, such that, 
for any $\varphi_n \in F_J$, we have
\begin{equation}\label{eqn:approximation_ftau}
\begin{aligned}
 & \int_{\Omega_x} \varphi_n \sum_{m=0}^{N} a_{nm} \partial_t u_{\tau,m} dx \\
+& \int_{\Omega_x} 
   \sqrt{\frac{T}2}
   \varphi_n 
   \sum_{m=0}^{N}
   \left( \sqrt{m}  a_{n,m-1}
	+ \sqrt{m+1}a_{n,m+1}
   \right) \partial_x u_{\tau,m} dx \\
-& \int_{\Omega_x}
   \sqrt{ \frac{2}T} E_{\tau}
   \varphi_n 
   \sum_{m=0}^{N} \sqrt{m+1} a_{n,m+1} u_{\tau,m} dx
=0,
\quad n = 0,\dots,N.
\end{aligned}
\end{equation}
The system above can also be written as
\begin{equation}\label{eqn:approximation_ftau_matrix}
\begin{aligned}
   \int_{\Omega_x} \left\langle A \partial_t U_{\tau}, \Phi \right\rangle dx
+ &\int_{\Omega_x} \sqrt{\dfrac{T}{2}} \left\langle A B \partial_x U_{\tau}, \Phi \right\rangle dx \\
- &\int_{\Omega_x} \sqrt{\dfrac{2}{T}} E_{\tau} \left\langle A D U_{\tau}, \Phi \right\rangle dx 
= 0,
\quad U_{\tau},\Phi \in \,\operatorname{supp}(F_J).
\end{aligned}
\end{equation}
Therefore, the approximation solution of the VP system \eqref{eqn:vlasov_poisson_system} 
obtained using the AW Hermite expansion in velocity and the Fourier expansion in space is defined by
\begin{equation}\label{eqn:expansion_ftau}
f_{\tau}(t,x,v) = \sum_{n=0}^N u_{\tau,n}(t,x) \psi_n(v),
\end{equation}
where $\tau = (1/J, 1/N)$ is a parameter, 
$(u_{\tau,n})_{0\leq n \leq N}$ satisfy \eqref{eqn:approximation_ftau} 
and $(\psi_n(v))_{0\leq n \leq N}$ are AW Hermite functions.

We now deal with the approximation $E_{\tau}$ of the electric field. 
We consider the potential function $\Psi_{\tau}$, such that 
\[
\left\{
\begin{aligned}
& E_\tau = -\frac{\partial \Psi_\tau}{\partial x} \\
& \dfrac{\partial E_\tau}{\partial x} = u_{\tau,0} - \rho_{0,\tau}.
\end{aligned}
\right.
\]
Hence we get the one dimensional Poisson equation 
\[
-\dfrac{\partial^2 \Psi_\tau}{\partial x^2} = u_{\tau,0} - \rho_{0,\tau},
\]
with $\rho_{0,\tau}$ such that 
\[
\int_{\Omega_x} (u_{\tau,0} - \rho_{0,\tau}) dx = 0.
\]

\section{Convergence}\label{sec:convergence}

In this section, we establish the convergence of the numerical solution $f_\tau$ given by
\eqref{eqn:approximation_ftau} and \eqref{eqn:expansion_ftau} 
to the solution $f$ to the VP system \eqref{eqn:vlasov_poisson_system},
which consists of a projection error and a consistency error.

\subsection{Global projection error}

First, we provide a global projection error on the combination of 
AW Hermite approximations and Fourier approximations in the suitable functional space.

We introduce $W_\tau$, the subspace definde by
\[
W_\tau := V_N^{\omega} \otimes F_J.
\]
For $f \in V^{\omega}$, it can be expanded as 
\[
f(\cdot,x,v) = \sum_{n \in \mathbb{N}} \widehat{u}_n(x) \psi_n(v).
\]
Then the projection $\mathcal{S}_\tau$ is given by 
\[
  \mathcal{S}_\tau f 
= \sum_{n=0}^N \mathcal{Q} 
  \left( \sum_{m=0}^\infty(P_p)_{nm} \widehat{u}_m(x) \right)
  \psi_n(v).
\]
\begin{Theorem}\label{theo:global_projection_error}
For all $f \in H_{\text{per}}^r (\omega dvdx)$, 
we have the global projection error
\[
\left|\left| f - \mathcal{S}_\tau f \right|\right|_{L^2(\omega_M dvdx)} 
\leq C \left( \left( \dfrac{N}{T} \right)^{-\frac{r}{2}} + J^{-s} \right) 
\left|\left| f \right|\right|_{H_{\text{per}}^r (\omega dvdx)}. 
\]
\end{Theorem}
\begin{proof}
We have, for all $f \in L^2(\omega dvdx)$,
\[
  f - \mathcal{S}_\tau f 
= \left( f - \mathcal{P} f \right) + 
  \left( \mathcal{P} f - \mathcal{S}_\tau f \right).
\]
On the one hand, observing the first term of the right-hand-side, 
we can estimate it thanks to 
Theorem \ref{theo:hermite_non_ortho_projection_error}
\[
|| f - \mathcal{P} f ||_{L^2(\omega_M dvdx)} 
\leq C \left( \dfrac{N}{T} \right)^{-\frac{r}{2}} 
|| f ||_{H_{\text{per}}^r (\omega dvdx)}. 
\]
For the second term, 
applying Cauchy-Schwarz inequality, and 
Proposition \ref{prop:fourier_projection_error}, we have
\[
\begin{aligned}
   \left|\left| \mathcal{P} f - \mathcal{Q}\,\mathcal{P} f \right|\right|_{L^2(\omega_M dvdx)}^2 
&= \left|\left| \left(\mathcal{I} - \mathcal{Q}\right)\mathcal{P} f \right|\right|_{L^2(\omega_M dvdx)}^2 \\ 
&\leq   
   \left|\left| \left(\mathcal{I} - \mathcal{Q}\right)\mathcal{P} f \right|\right|_{L^2(\omega   dvdx)}^2 \\ 
&= \sum_{n=0}^{N} \left|\left| 
   \left(\mathcal{I} - \mathcal{Q}\right) \widehat{u}_n
   \right|\right|_{L^2(dx)}^2 \\ 
&\leq    
   \sum_{n=0}^{\infty} \left( C J^{-s} \left| \widehat{u}_n \right|_{H_{\text{per}}^s(dx)} \right)^2 \\
&= \left(C J^{-s} \right)^2 \left|\left| f \right|\right|_{H_{\text{per}}^s(dx)}^2.
\end{aligned}
\]
Gathering the latter results and using $r \geq s$, we establish that 
there exists a constant $C > 0$, independent of the discretization parameter $\tau = \left( 1/J, 1/N \right)$, 
such that
\[
\left|\left| f - \mathcal{S}_\tau f \right|\right|_{L^2(\omega_M dvdx)} 
\leq C \left( \left( \dfrac{N}{T} \right)^{-\frac{r}{2}} + J^{-s} \right) 
\left|\left| f \right|\right|_{H_{\text{per}}^r (\omega dvdx)}. 
\]
\end{proof}

\subsection{Consistency error}

In the following, we estimate the consistency error.
\begin{proposition}\label{prop:consistency_error}
For the numerical solution $f_\tau$ given by
\eqref{eqn:approximation_ftau} and \eqref{eqn:expansion_ftau}, 
and the solution $f$ to the VP system \eqref{eqn:vlasov_poisson_system},
for $t \geq 0$, the consistency error satisfies
\[
\left|\left| \mathcal{S}_\tau f(t) - f_\tau(t)\right|\right|_{L^2(\omega_M dvdx)}
\leq 
  e^{C\sqrt M t}
  \left(
  \left( \dfrac{N}{T} \right)^{-\frac{r-1}{2}} + J^{-(s-1)}
  \right),
  \quad M\ge 1.
\]
\end{proposition}
\begin{proof}
For convenience, we give the following notations. 
To begin with, $U_{\tau,1} = (u_{\tau,n})_{0\leq n \leq N}$ corresponds to $f_\tau$ satisfying the 
discrete system \eqref{eqn:approximation_ftau}. 
In addition, 
$\widehat{U}_1 = (\widehat{u}_n)_{0\leq n \leq N}$, 
$\widehat{U}_2 = (\widehat{u}_n)_{n > N}$ correspond to the exact continuous solution $f$ 
of the VP system \eqref{eqn:vlasov_poisson_system}.
Next,
we consider the projection of the exact continuous solution $f$
\[
  \mathcal{S}_\tau f 
= \sum_{n=0}^N \mathcal{Q} 
  \left( \sum_{n=0}^\infty (P_p)_{nm} \widehat{u}_m \right)
  \psi_n,
\]
and let $\widetilde{U}_1$ denote $(\mathcal{Q} \left( \sum_{n=0}^\infty (P_p)_{nm} \widehat{u}_m \right))_{0\leq n \leq N}$.

From the notations above, we introduce $(\Delta_{\tau,n})_{n\in \mathbb{N}}$ 
and $(\Gamma_{\tau,n})_{n\in \mathbb{N}}$ as
\[
\Delta_\tau = 
\begin{pmatrix}
\widetilde{U}_1 - U_{\tau,1} \\
0
\end{pmatrix}
\]
and
\[
\Gamma_\tau = 
\begin{pmatrix}
\widehat{U}_1 - \widetilde{U}_1 \\
\widehat{U}_2
\end{pmatrix}.
\]
Since $\Delta_\tau \in V_N^{\omega} \otimes F_J$, it satisfies the discrete system \eqref{eqn:approximation_ftau}. 
Therefore, we have
\[
\int_{\Omega_x}
\Big\langle A \left( 
\partial_t U_\tau + B \partial_x U_\tau - E_\tau D U_\tau
\right), \Delta_\tau \Big\rangle 
\,dx = 0.
\]
We also have
\[
  \int_{\Omega_x}
  \,\Big\langle
  \partial_t \widehat{U} + B \partial_x \widehat{U} - E D \widehat{U}, 
  A \Delta_\tau \Big\rangle 
  \,dx
= 0.
\]
Substracting two systems above, we have 
\[
\begin{aligned}
 &\int_{\Omega_x}
  \,\Big\langle
  \partial_t \Delta_\tau + B \partial_x \Delta_\tau - E_\tau D \Delta_\tau, 
  A \Delta_\tau \Big\rangle 
  \,dx\\
+&\int_{\Omega_x}
  \,\Big\langle
  \partial_t \Gamma_\tau + B \partial_x \Gamma_\tau - E_\tau D \Gamma_\tau 
+ \left( E_\tau - E \right) D \widehat{U}, 
  A \Delta_\tau \Big\rangle 
  \,dx\\
=&\,0.
\end{aligned}
\]
We aim to estimate the consistency error by 
\[
  \left|\left| \mathcal{S}_\tau f(t) - f_\tau(t)\right|\right|_{L^2(\omega_M dvdx)}^2 
= \int_{\Omega_x} \Big\langle A \Delta_\tau, \Delta_\tau \Big\rangle \,dx.
\]
We compute the time derivative of the equality above, given by
\[
  \dfrac{1}{2} \dfrac{\mathrm{d}}{\mathrm{d t}}
  \left|\left| \mathcal{S}_\tau f(t) - f_\tau(t)\right|\right|_{L^2(\omega_M dvdx)}^2 
= \int_{\Omega_x} \Big\langle A \partial_t \Delta_\tau, \Delta_\tau \Big\rangle \,dx.
\]
Hence we obtain
\begin{equation}\label{eqn:time_derivative_consistency_error}
  \dfrac{1}{2} \dfrac{\mathrm{d}}{\mathrm{d t}}
  \left|\left| \mathcal{S}_\tau f(t) - f_\tau(t)\right|\right|_{L^2(\omega_M dvdx)}^2 
= J_1 + J_2 + J_{3,1} + J_{3,2} + J_{3,3},
\end{equation}
where
\[
\begin{aligned}
   J_1 
&= \int_{\Omega_x} \Big\langle
   - A B \partial_x \Delta_\tau + E_\tau A D \Delta_\tau, 
   \Delta_\tau \Big\rangle \,dx, \\
   J_2 
&= \int_{\Omega_x} \Big\langle
   - (E_\tau-E) D \widehat{U}, 
   A \Delta_\tau \Big\rangle \,dx, \\
   J_{3,1} 
&= \int_{\Omega_x} \Big\langle
   - \partial_t \Gamma_\tau, 
   A \Delta_\tau \Big\rangle \,dx, \\
   J_{3,2} 
&= \int_{\Omega_x} \Big\langle
   - E_\tau D \Gamma_\tau, 
   A \Delta_\tau \Big\rangle \,dx, \\
   J_{3,3} 
&= \int_{\Omega_x} \Big\langle
   - B \partial_x \Gamma_\tau, 
   A \Delta_\tau \Big\rangle \,dx. \\
\end{aligned}
\]
Let us estimate these terms one by one. 
On the one hand, we estimate $J_1$, 
which we refer to as \textit{the leading term} of the consistency error,
since it is associated with the vector $\Delta_\tau$.
It is straightforward to verify that the transport term vanishes. 
Indeed, we have
\[
  \int_{\Omega_x} \Big\langle
  A B \partial_x \Delta_\tau, 
  \Delta_\tau \Big\rangle \,dx 
= \dfrac{1}{2}
  \int_{\Omega_x} \partial_x \Big\langle
  A B \Delta_\tau, 
  \Delta_\tau \Big\rangle \,dx = 0. 
\]
For the acceleration term, we have
\[
\begin{aligned}
  \left|
  \int_{\Omega_x} \Big\langle
  E_\tau A D \Delta_\tau, 
  \Delta_\tau \Big\rangle \,dx
  \right|
&\leq 
  \left|
  \int_{\Omega_x} E_\tau 
  \int_{\mathbb{R}} 
  (\mathcal{S}_\tau f - f_\tau)
  \partial_v(\mathcal{S}_\tau f - f_\tau)
  \omega_M dv
  \,dx
  \right| \\
&=
  \left|
  \int_{\Omega_x} E_\tau 
  \int_{\mathbb{R}} 
  (\mathcal{S}_\tau f - f_\tau)^2
  \omega_M^\prime dv
  \,dx
  \right| \\
&\leq
  \int_{\Omega_x} |E_\tau| 
  \int_{\mathbb{R}} 
  (\mathcal{S}_\tau f - f_\tau)^2
  |\omega_M^\prime| dv
  \,dx \\
&\leq
  2 \sqrt{M} C
  \| E_\tau \|_{L^\infty(dx)} \|
  \| \mathcal{S}_\tau f - f_\tau \|^2_{L^2(\omega_M dvdx)}.
\end{aligned}
\]
Because $J_1$ decomposes into the sum of the transport and acceleration terms, 
it follows 
\[
  |J_1|
\leq
  2 \sqrt{M} C
  \| E_\tau \|_{L^\infty(dx)} \|
  \| \mathcal{S}_\tau f - f_\tau \|^2_{L^2(\omega_M dvdx)}.
\]

On the other hand, we estimate $J_2, J_{3,1}, J_{3,2}, J_{3,3}$, 
which we refer to as \textit{the projection term} of the consistency error,
since they are associated with the vectors $\Gamma_\tau$ and $\widehat{U}$.
As the results presented here are closely similar to 
those established in \cite[Section 4]{bessemoulin2023convergence}, 
we provide only a concise outline of the proofs.

We first estimate $J_2$.
The estimate process is very similar to that in~\cite[Page 1681]{bessemoulin2023convergence}.
We write
\[
\begin{aligned}
  |J_2|
&\leq C 
  \| E - E_\tau \|_{L^\infty(dx)}
  \| \partial_v f \|_{L^2(\omega_M dvdx)}
  \| \mathcal{S}_\tau f - f_\tau \|_{L^2(\omega_M dvdx)} \\
&\leq C 
  \| f  - f_\tau \|_{L^2(\omega_M dvdx)}
  \| \mathcal{S}_\tau f - f_\tau \|_{L^2(\omega_M dvdx)},
  \quad M > 0. \\
\end{aligned}
\]
Because the total error decomposes into the sum of the projection error and the consistency error, 
it follows by an application of the triangle inequality 
combined with Theorem \ref{theo:hermite_non_ortho_projection_error} that
\[
|J_2| 
  \leq C 
  \left(
  \left( \dfrac{N}{T} \right)^{-\frac{r}{2}} + J^{-s}
+ \left| \left| \mathcal{S}_\tau f -f_\tau \right|\right|_{L^2(\omega_M dvdx)}
  \right)
  \left| \left| \mathcal{S}_\tau f -f_\tau \right|\right|_{L^2(\omega_M dvdx)}.
\]
It then follows from Young's inequality that
\[
|J_2| 
  \leq C 
  \left(
  \left(
  \left( \dfrac{N}{T} \right)^{-\frac{r}{2}} + J^{-s}
  \right)^2
+ \left| \left| \mathcal{S}_\tau f -f_\tau \right|\right|_{L^2(\omega_M dvdx)}^2
  \right).
\]
\begin{Remark}\label{remark:important_remark}
The density and electric-field estimates have the relation below:
\[
  \| E - E_\tau \|^2_{L^2(dx)} 
\leq
  C\,\| \widehat{u}_0 - u_{\tau,0} \|^2_{L^2(dx)} 
\]
if $M\ge 1$, we have
\[
  \| \widehat{u}_0 - u_{\tau,0} \|^2_{L^2(dx)} 
\leq
  C\,\| f - f_\tau \|^2_{L^2(\omega_M dvdx)}. 
\]
if $M=0$, 
since velocity integration is not a bounded functional on $L^2(\omega_M dv) = L^2(dv)$, 
so control of the distribution error in that norm alone does not yield 
the required density of electric-field estimate. 
Additional velocity control or a separate argument with constants uniform 
in the discretization parameters is necessary for the case $M=0$.

When $M=0$, the AW Hermite trial space coincides with a suitably rescaled standard SW Hermite space: 
both consist of polynomials of degree at most $N$ multiplied by the same Gaussian. 
With identical initial projection and field coupling, 
unweighted Galerkin approximation is independent of this choice of basis.
Consequently, the convergence analysis for the case \(M=0\) is identical to 
that established in~\cite{manzini2017convergence}, 
to which we refer the reader for further details.
\end{Remark}
Next, we estimate $J_{3,1}$.
We use the Cauchy-Schwarz inequality to $|J_{3,1}|$
\[
  |J_{3,1}|
\leq
  \left|
  \int_{\Omega_x} 
  \Big\langle
    \partial_t \Gamma_\tau,
  A \partial_t \Gamma_\tau 
  \Big\rangle \,dx
  \right|^{\frac12}
  \left|
  \int_{\Omega_x} 
  \Big\langle
    \Delta_\tau,
  A \Delta_\tau 
  \Big\rangle \,dx
  \right|^{\frac12}.
\]
We remark that
\[
\left|
\int_{\Omega_x} 
\Big\langle
  \partial_t \Gamma_\tau, 
A \partial_t \Gamma_\tau 
\Big\rangle \,dx
\right|^{\frac12}
\le
\|\partial_t f - \mathcal{S}_\tau \partial_t f \|_{L^2(\omega_M dv dx)}.
\]
Since $f(t) \in H^r(\omega dv dx)$ satisfies 
the Vlasov system~\eqref{eqn:vlasov_poisson_system}, 
and using the second estimate of~\cite[Lemma 2.2]{fok2002combined}, 
we have
\begin{equation}\label{eq:estimate_function_temporal}
\begin{aligned}
  \| \partial_t f\|_{H^{r-1}(\omega_M dv dx)}
&\le
  \|v\partial_x f\|_{H^{r-1}(\omega_M dv dx)}
+ \| \partial_v f\|_{H^{r-1}(\omega_M dv dx)} \|E\|_{L^\infty} \\ 
&\le
  \|v\partial_x f\|_{H^{r-1}(\omega   dv dx)}
+ \| \partial_v f\|_{H^{r-1}(\omega   dv dx)} \|E\|_{L^\infty} \\ 
&\le
  \dfrac{C}{T}\|f\|_{H^{ r }(\omega   dv dx)}.
\end{aligned}
\end{equation}
We use Theorem~\ref{theo:global_projection_error} to $\partial_t f$ 
and insert the inequality~\eqref{eq:estimate_function_temporal}
\[
\begin{aligned}
  \|\partial_t f - \mathcal{S}_\tau \partial_t f \|_{L^2(\omega_M dv dx)}
&\leq C
  \left(
  \left( \dfrac{N}{T} \right)^{-\frac{r-1}{2}} + J^{-s}
  \right)
  \| \partial_t f \|_{H^{r-1}(\omega_M dvdx)} \\
&\leq C
  \left(
  \left( \dfrac{N}{T} \right)^{-\frac{r-1}{2}} + J^{-s}
  \right)
  \| f \|_{H^{r}(\omega dvdx)}. \\
\end{aligned}
\]
Thus, we have
\[
|J_{3,1}| 
  \leq C 
  \left(
  \left( \dfrac{N}{T} \right)^{-\frac{r-1}{2}} + J^{-s}
  \right)
  \left| \left| \mathcal{S}_\tau f -f_\tau \right|\right|_{L^2(\omega_M dvdx)}.
\]
In the same manner, an application of Young’s inequality leads to
\[
|J_{3,1}| 
  \leq C 
  \left(
  \left(
  \left( \dfrac{N}{T} \right)^{-\frac{r-1}{2}} + J^{-s}
  \right)^2
+ \left| \left| \mathcal{S}_\tau f -f_\tau \right|\right|_{L^2(\omega_M dvdx)}^2
  \right).
\]

Thirdly, we estimate $|J_{3,2}|$. 
The estimate on $|J_{3,2}|$ follows the same arguements 
as the estimate for $|J_{3,1}|$. 
We use H\"older inequality and Cauchy-Schwarz inequality to $|J_{3,2}|$
\[
  |J_{3,2}|
\leq
  \| E_\tau \|_{L^\infty(dx)}
  \left|
  \int_{\Omega_x} 
  \Big\langle
    D \Gamma_\tau,
  A D \Gamma_\tau 
  \Big\rangle \,dx
  \right|^{\frac12}
  \left|
  \int_{\Omega_x} 
  \Big\langle
    \Delta_\tau,
  A \Delta_\tau 
  \Big\rangle \,dx
  \right|^{\frac12}.
\]
Indeed, we remark again that
\[
\left|
\int_{\Omega_x} 
\Big\langle
  D \Gamma_\tau, 
A D \Gamma_\tau 
\Big\rangle \,dx
\right|^{\frac12}
\le
\|\partial_v f - \mathcal{S}_\tau \partial_v f \|_{L^2(\omega_M dv dx)}.
\]
We use Theorem \ref{theo:global_projection_error} to $\partial_v f$ and 
we obtain 
\[
|J_{3,2}| 
  \leq C 
  \left(
  \left( \dfrac{N}{T} \right)^{-\frac{r-1}{2}} + J^{-s}
  \right)
  \left| \left| \mathcal{S}_\tau f -f_\tau \right|\right|_{L^2(\omega_M dvdx)}.
\]
Likewise, applying Young's inequality yields
\[
|J_{3,2}| 
  \leq C 
  \left(
  \left(
  \left( \dfrac{N}{T} \right)^{-\frac{r-1}{2}} + J^{-s}
  \right)^2
+ \left| \left| \mathcal{S}_\tau f -f_\tau \right|\right|_{L^2(\omega_M dvdx)}^2
  \right).
\]

For the last term $J_{3,3}$, we use Cauchy-Schwarz inequality and 
we obtain
\[
|J_{3,3}| 
\leq
  \left|
  \int_{\Omega_x} 
  \Big\langle
    B\partial_x \Gamma_\tau,
  A B\partial_x \Gamma_\tau
  \Big\rangle \,dx
  \right|^{\frac12}
  \left|
  \int_{\Omega_x} 
  \Big\langle
    \Delta_\tau,
  A \Delta_\tau 
  \Big\rangle \,dx
  \right|^{\frac12}.
\]
We write
\[
  |J_{3,3}| 
\leq C 
  \| \,v\,\partial_x (f - \mathcal{S}_\tau f) \|_{L^2(\omega_M dvdx)}
  \| f_\tau - \mathcal{S}_\tau f \|_{L^2(\omega_M dvdx)}
\]
Using Lemma~\ref{lem:lemma-non-orthogonal-projection}, we obatain
\[
|J_{3,3}| 
  \leq C 
  \left(
  \left( \dfrac{N}{T} \right)^{-\frac{r-1}{2}} + J^{-(s-1)}
  \right)
  \left| \left| \mathcal{S}_\tau f -f_\tau \right|\right|_{L^2(\omega_M dvdx)}.
\]
Likewise, applying Young's inequality yields
\[
|J_{3,3}| 
  \leq C 
  \left(
  \left(
  \left( \dfrac{N}{T} \right)^{-\frac{r-1}{2}} + J^{-(s-1)}
  \right)^2
+ \left| \left| \mathcal{S}_\tau f -f_\tau \right|\right|_{L^2(\omega_M dvdx)}^2
  \right).
\]

Finally, 
collecting all the estimates in \eqref{eqn:time_derivative_consistency_error}, 
this yields
\[
\begin{aligned}
& \dfrac{1}{2} \dfrac{\mathrm{d t}}{\mathrm{d}}
  \left|\left| \mathcal{S}_\tau f(t) - f_\tau(t)\right|\right|_{L^2(\omega_M dvdx)}^2 \\ 
& \leq C \sqrt{M} 
  \left(
  \left| \left| \mathcal{S}_\tau f -f_\tau \right|\right|_{L^2(\omega_M dvdx)}^2
+ \left(
  \left( \dfrac{N}{T} \right)^{-\frac{r-1}{2}} + J^{-(s-1)}
  \right)^2
  \right).
\end{aligned}
\]
We conclude using Gronwall's lemma that, $\forall t \geq 0$, we have
\[
\begin{aligned}
\left|\left| \mathcal{S}_\tau f(t) - f_\tau(t)\right|\right|_{L^2(\omega_M dvdx)}^2 
& \leq 
  \left(e^{2C\sqrt M t}-1\right)
  \left(
  \left( \dfrac{N}{T} \right)^{-\frac{r-1}{2}} + J^{-(s-1)}
  \right)^2 \\
& \leq 
  e^{2C\sqrt M t}
  \left(
  \left( \dfrac{N}{T} \right)^{-\frac{r-1}{2}} + J^{-(s-1)}
  \right)^2,
\end{aligned}
\]
which completes the proof.
\end{proof}

\subsection{Final result}

From the preivous estimates, we are now ready to establish the final result on the 
convergence of the numerical solution $f_\tau$ given by
\eqref{eqn:approximation_ftau} and \eqref{eqn:expansion_ftau}, 
and the solution $f$ to the VP system \eqref{eqn:vlasov_poisson_system},
for $t \geq 0$. 
\begin{Theorem}
$\forall t \geq 0$, let $f(t,\cdot) \in H^{r}(\omega dvdx)$ be the solution of 
the VP system \eqref{eqn:vlasov_poisson_system} where $r \geq s$ and 
$f_\tau$ be the approximation of \eqref{eqn:approximation_ftau} and \eqref{eqn:expansion_ftau}.
Then there exists two constants $C_1, C_2 > 0$, 
independent of $\tau = \left( 1/J, 1/N \right)$, such that
\[
\left|\left| f(t) - f_\tau(t)\right|\right|_{L^2(\omega_M dvdx)}
\leq 
  \left( e^{C_1\sqrt M t}+C_2 \right)
  \left(
  \left( \dfrac{N}{T} \right)^{-\frac{r-1}{2}} + J^{-(s-1)}
  \right),
  \quad M \ge 1.
\]
\end{Theorem}
\begin{proof}
By the triangle inequality, we have
\[
\left|\left|
f(t) - f_\tau(t)
\right|\right|_{L^2(\omega_M dvdx)}
\leq
  || f - \mathcal{S}_\tau f ||_{L^2(\omega_M dvdx)}
+ \left|\left| \mathcal{S}_\tau f(t) - f_\tau(t)\right|\right|_{L^2(\omega_M dvdx)},
\]
where the first term on the right-hand side is the global projection error 
and the second one is the consistency error.
According to Theorem \ref{theo:global_projection_error}
and Proposition \ref{prop:consistency_error}, we complete the proof.
\end{proof}

\section{Conclusion}\label{sec:conclusion}

In this article, we investigate the convergence of 
Fourier/Hermite spectral methods with polynomial weights in velocity
based on AW Hermite/Fourier functions for the VP system.
In particular, 
we introduce a suitable weighted norm with a time-independent polynomial weight,
which leads to the stability of the numerical method.
For the convergence result, 
we have derived an estimate for the consistency error 
between the numerical solution of the VP system and its global projection. 
A key observation is that, 
owing to the structure of the AW Hermite/Fourier discretized system 
within the framework of the Fourier/Hermite spectral methods with polynomial weights in velocity,
the leading term $J_1$ of the consistency error is controlled. 
This property plays a role in obtaining the estimate of the consistency error,
consequently, in establishing the convergence result of the proposed spectral method.
Although, our analysis is an one dimensional case, 
The proposed spectral method can be extended straightforwardly 
to 2D and 3D velocity spaces.
Therefore, in a future work, we would like to adapt our approach 
to the multi-dimensional case following the idea in this work. 
For numerical experiments, we would discuss numerical robustness and performance in the future. 
Accuracy-versus-cost comparisons, and tests of conditioning, 
conservation errors, and longer-time behavior will be analyzed to substantiate the practical contribution.


\end{document}